\documentclass[12pt]{amsart}
\usepackage{graphicx}
\newtheorem{theorem}{Theorem}[section]
\newtheorem{proposition}[theorem]{Proposition}

\newtheorem{corollary}[theorem]{Corollary}

\theoremstyle{remark}

\newtheorem{definition}[theorem]{Definition}

\numberwithin{equation}{section}

\begin{document}

\title[
Minimal unknotting sequences of Reidemeister moves 
]{
Minimal unknotting sequences of Reidemeister moves
containing unmatched RII moves
}

\author{Chuichiro Hayashi, Miwa Hayashi, Minori Sawada and Sayaka Yamada}

\date{\today}

\thanks{The first author is partially supported
by Grant-in-Aid for Scientific Research (No. 22540101),
Ministry of Education, Science, Sports and Technology, Japan.}

\begin{abstract}
 Arnold introduced invariants $J^+$, $J^-$ and $St$
for generic planar curves. 
 It is known that both $J^+ /2 + St$ and $J^- /2 + St$ are invariants
for generic spherical curves.
 Applying these invariants to underlying curves of knot diagrams,
we can obtain lower bounds for the number of Reidemeister moves for uknotting.
 $J^- /2 + St$ works well for unmatched RII moves.
 However, it works only by halves for RI moves.
 Let $w$ denote the writhe for a knot diagram.
 We show that $J^- /2 + St \pm w/2$ works well also for RI moves,
and demonstrate that it gives a precise estimation 
for a certain knot diagram of the unknot
with the underlying curve 
$r = 2 + \cos (n \theta/(n+1)),\ (0 \le \theta \le 2(n+1)\pi$).
\\
{\it Mathematics Subject Classification 2010:}$\ $ 57M25.\\
{\it Keywords:}$\ $
knot diagram, Reidemeister move, Arnold invariant, writhe
\end{abstract}

\maketitle

\section{Introduction}

 In this paper, all the knots are assumed to be oriented.
 A {\it Reidemeister move} is a local move of a knot diagram
as in Figure \ref{fig:Reid123}.
 An RI (resp. II) move
creates or deletes a monogon face (resp. a bigon face).
 An RII move is called {\it matched} or {\it unmatched} with respect to the orientation of the knot
as shown in Figure \ref{fig:matched}.
 An RIII move is performed on a $3$-gon face,
deleting it and creating a new one.
 Any such move does not change the knot type.
 As Alexander and Briggs \cite{AB} and Reidemeister \cite{R} showed,
for any pair of diagrams $D_1$, $D_2$ which represent the same knot type,
there is a finite sequence of Reidemeister moves
which deforms $D_1$ to $D_2$.

\begin{figure}[htbp]
\begin{center}
\includegraphics[width=10cm]{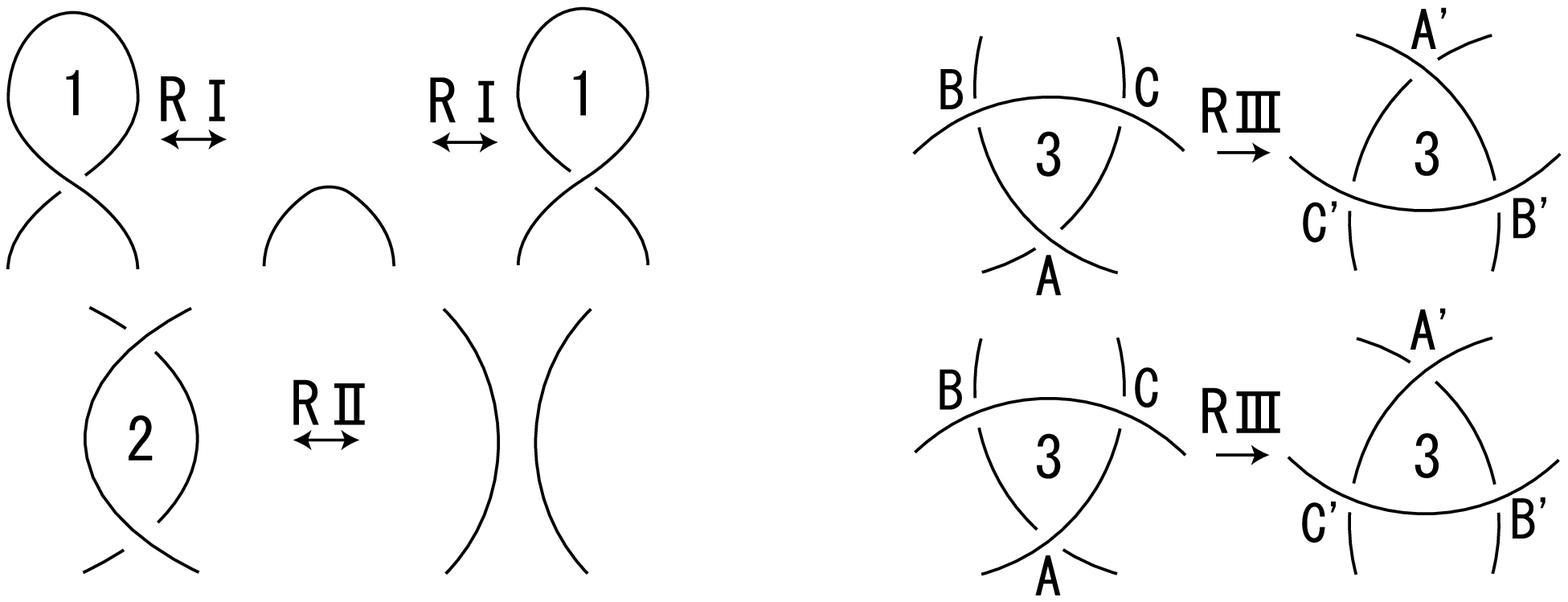}
\end{center}
\caption{}
\label{fig:Reid123}
\end{figure} 

\begin{figure}[htbp]
\begin{center}
\includegraphics[width=5cm]{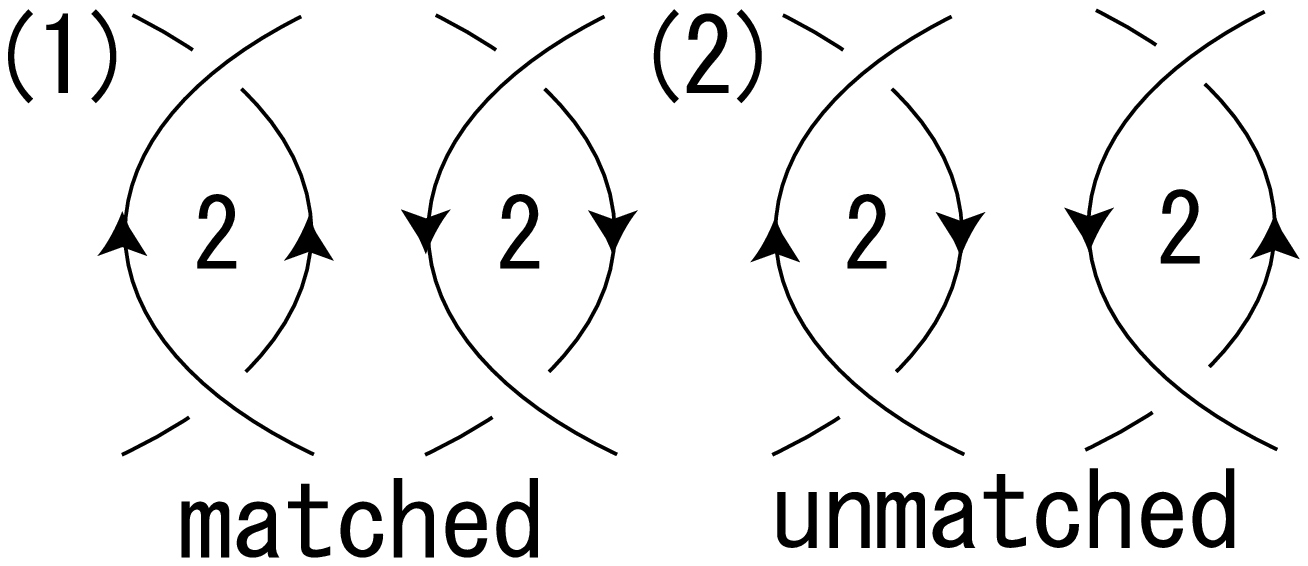}
\end{center}
\caption{}
\label{fig:matched}
\end{figure} 

 Necessity of Reidemeister moves of type II and III is studied 
in \cite{O}, \cite{M} and \cite{Hg}.
 In \cite{H}, 
the knot diagram invariant cowrithe is introduced,
and it gives a lower bound for the number of matched RII and RIII moves.
 In \cite{CESS},
Carter, Elhamdadi, Saito and Satoh
gave a lower bound for the number of RIII moves
by using extended n-colorings of knot diagrams in ${\Bbb R}^2$.
 Hass and Nowik
introduced a certain knot diagram invariant
by using smoothing and linking number
in \cite{HN1},
and gave 
in \cite{HN2}
an example of an infinite sequence
of diagrams of the trivial knot
such that 
the $n$-th one
has $7n-1$ crossings,
can be unknotted by $2n^2+3n$ Reidemeister moves,
and needs at least $2n^2+3n-2$ Reidemeister moves 
for being unknotted.
 Using cowrithe, it is shown in \cite{HH}
that a certain sequence of Reidemeister moves
bringing $D(n+1,n)$ to $D(n,n+1)$ is minimal,
where $D(p,q)$ denotes the usual diagram of the $(p,q)$-torus knot.
 In the above papers \cite{HN2} and \cite{HH},
the sequences of Reidemeister moves do not contain unmatched RII moves.
 It is not easy
to estimate the number of unmatched RII moves needed for unknotting.
 In this paper, we show
that a certain unknotting sequence of Reidemeister moves
containing unmatched RII moves
is minimal,
using the writhe 
and the Arnold invariants of the underlying spherical curve,
the knot diagram with 
over-under informations at the crossings forgotten.

\begin{figure}[htbp]
\begin{center}
\includegraphics[width=10cm]{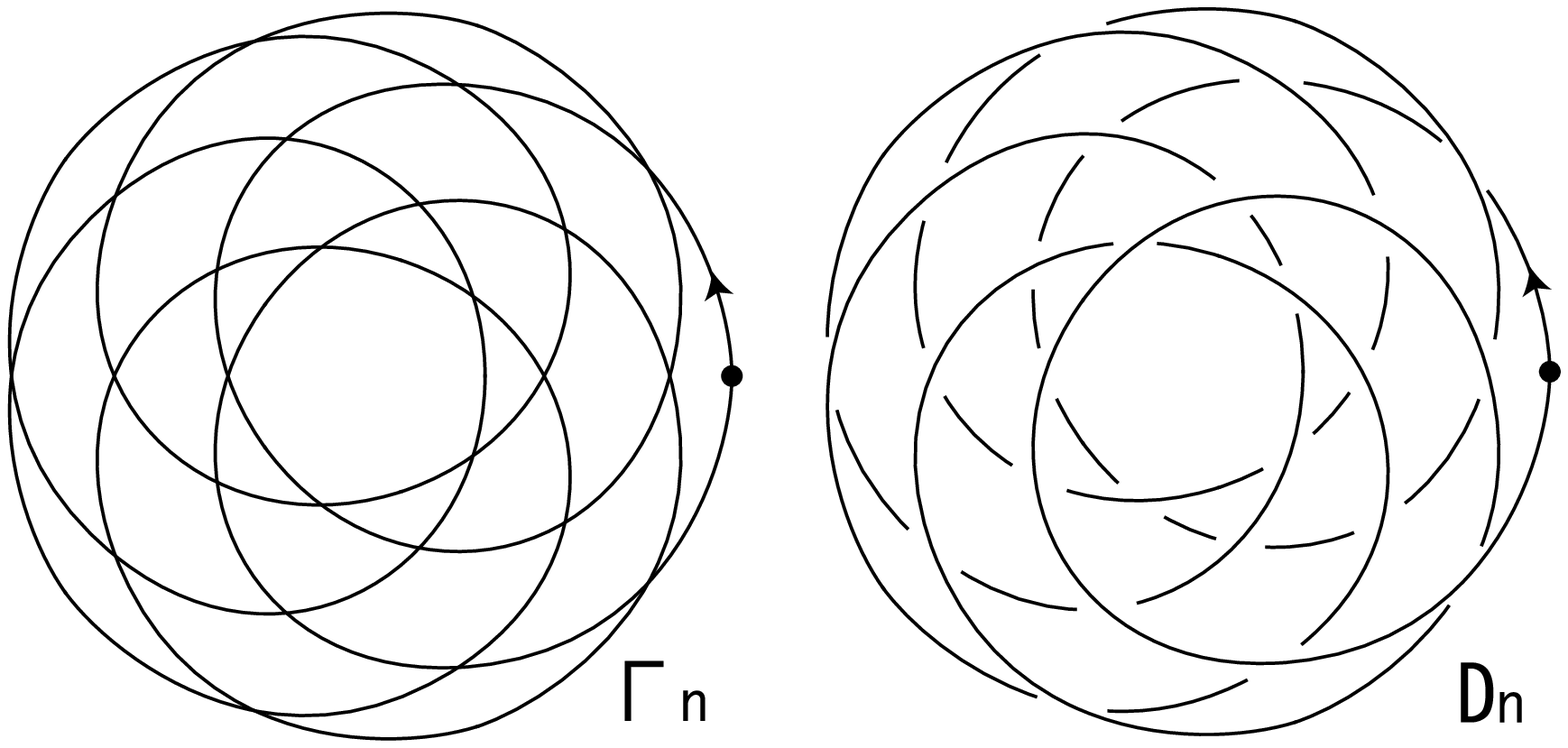}
\end{center}
\caption{}
\label{fig:Gamma_n}
\end{figure}

 Let $n$ be an integer larger than or equal to $2$.
 As the underlying spherical curve of a knot diagram,
we consider $\Gamma_n$ as shown in Figure \ref{fig:Gamma_n}, where $n=5$. 
 We regard the $2$-sphere $S^2$ as ${\Bbb R}^2 \cup \{ \infty \}$.
 For an integer $n$ larger than or equal to $2$,
$\Gamma_n$ is given by the equation 
$r = 2 + \cos (n \theta/(n+1))$, $(0 \le \theta \le 2(n+1)\pi)$
with respect to the polar coordinates $(r, \theta)$ on the plane.
 The curve $\Gamma_n$ has an $n$-gonal face at center,
surrounded by a cycle of $n$ trigonnal faces
surrounded by $n-2$ cycles of $n$ quadrilateral faces
surrounded by a cycle of $n$ trigonnal faces.
 The outermost region of $\Gamma_n$ is an $n$-gonal face.
 We set the base point $p$ to be $(r,\theta)=(3,0)=(3,2(n+1)\pi)$,
and give $\Gamma_n$ an orientation 
in the direction of which $\theta$ increases.
 The knot diagram $D_n$ is obtained from $\Gamma_n$
by giving 
over-under informations at all double points
so that they are ascending as below.
 Every crossing is composed of two subarcs of the knot.
 When we travel along the knot, 
staring at the base point
and going in the direction of the orientation,
we meet the first subarc and then the second one.
 In the diagram $D_n$
the second subarc goes over the first one.
 Thus $D_n$ represents the trivial knot.

 This knot diagram $D_n$ is also obtained 
from the usual diagram of the $(n+1, n)$-torus knot $T(n+1,n)$
by changing crossings
so that $D_n$ is ascending.
 The usual diagram of $T(n+1,n)$ is the closure of the $(n+1)$-braid
$(\sigma_1^{-1}\sigma_2^{-1}\cdots\sigma_n^{-1})^n$,
while $D_n$ is the closed braid of the $(n+1)$-braid below.
\begin{center}
$(\sigma_1^{-1}\sigma_2^{-1}
    \cdots\sigma_{n-2}^{-1}\sigma_{n-1}^{-1}\sigma_n^{-1})
 (\sigma_1^{-1}\sigma_2^{-1}
    \cdots\sigma_{n-2}^{-1}\sigma_{n-1}^{-1}\sigma_n)$
$\ \ \ \ \ \ \ \ \ \ \ $ \\
$\ \ \ \ \ \ \ \ \ \ \ $
$(\sigma_1^{-1}\sigma_2^{-1}
    \cdots\sigma_{n-2}^{-1}\sigma_{n-1}\sigma_n)
 \cdots
 (\sigma_1^{-1}\sigma_2
    \cdots\sigma_{n-2}\sigma_{n-1}\sigma_n)$
\end{center}
 In this braid,
the $j$-th strand goes over the $i$-th strand if $i < j$.

\begin{theorem}\label{theorem:Dn}
 For any integer $n$ larger than or equal to $3$,
the knot digram $D_n$ of the trivial knot
can be deformed to the trivial diagram with no crossing
by a sequence of $n(n^2+5)/6$ Reidemeister moves,
which consists of 
${}_n C_1 = n$ RI moves deleting positive crossings, 
${}_n C_2 = (n-1)n/2$ unmatched RII moves deleting bigons
and ${}_n C_3 = (n-2)(n-1)n/6$ positive RIII moves.
 Moreover, 
any sequence of Reidemeister moves bringing $D_n$ to the trivial diagram
must contain 
at least $n(n^2+5)/6$
RI moves deleting positive crossings,
unmatched RII move deleting bigons
or 
positive RIII moves.
 Hence, the above sequence is minimal.
\end{theorem}

 To prove the above theorem,
we use the knot diagram invariant $J^-/2 + St \pm w/2$,
where $J^-$ and $St$ are the Arnold invariants for plane curves,
and $w$ is the writhe.
 We consider the changes of this invariant under Reidemeister moves in Section 2
after recalling the definitions 
of the Arnold invariants.
 Theorem \ref{theorem:Dn} is proved in Section 3.


\section{knot diagram invariants}

 First, we recall the definition of the Arold invariants.
 A {\it plane curve} is a smooth immersion of the oriented circle $S^1$
to the plane ${\Bbb R}^2$.
 It is {\it generic}
if it has only a finite number of multiple points,
and they are transverse double points.

\begin{figure}[htbp]
\begin{center}
\includegraphics[width=11cm]{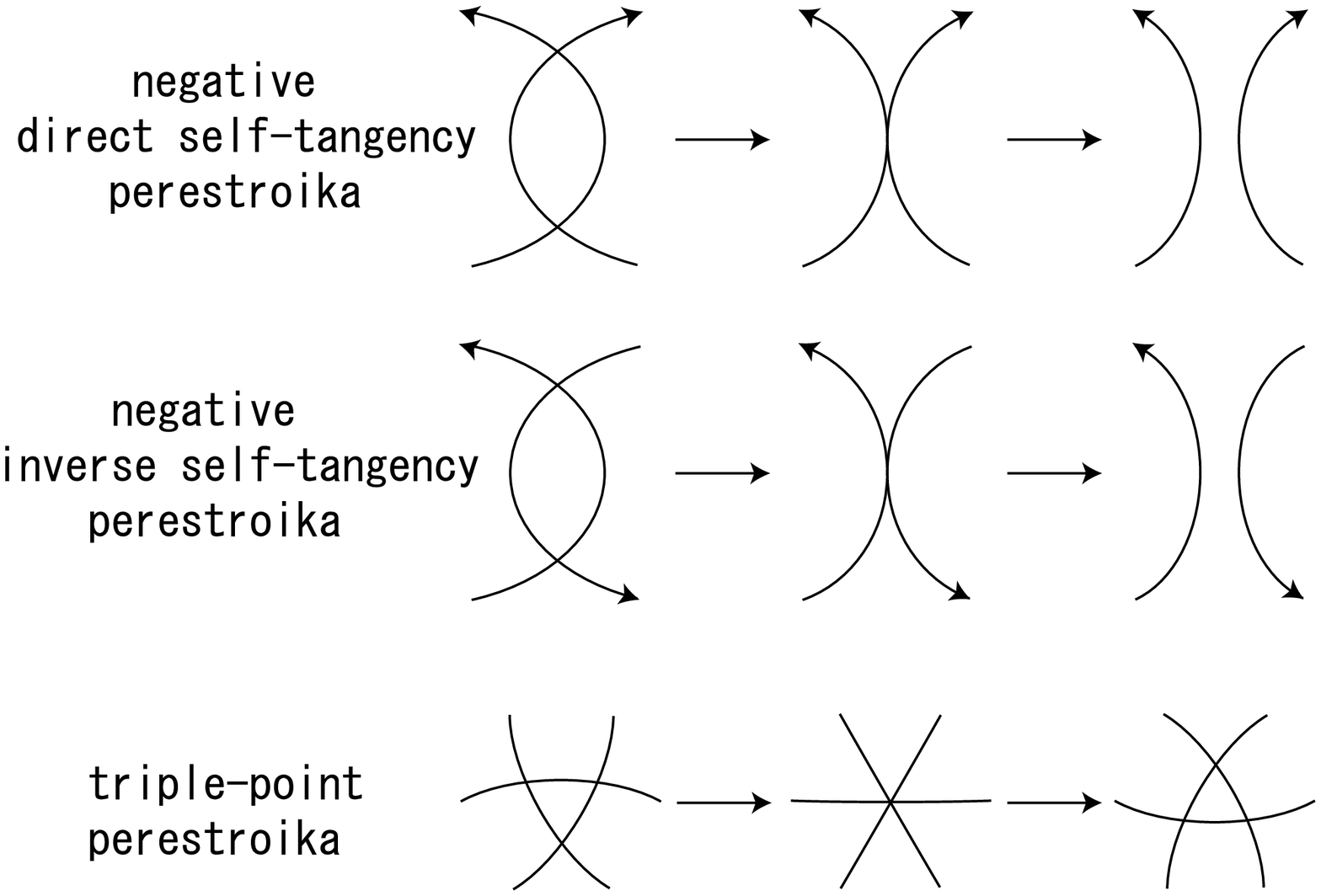}
\end{center}
\caption{}
\label{fig:perestroika}
\end{figure} 

\begin{figure}[htbp]
\begin{center}
\includegraphics[width=9cm]{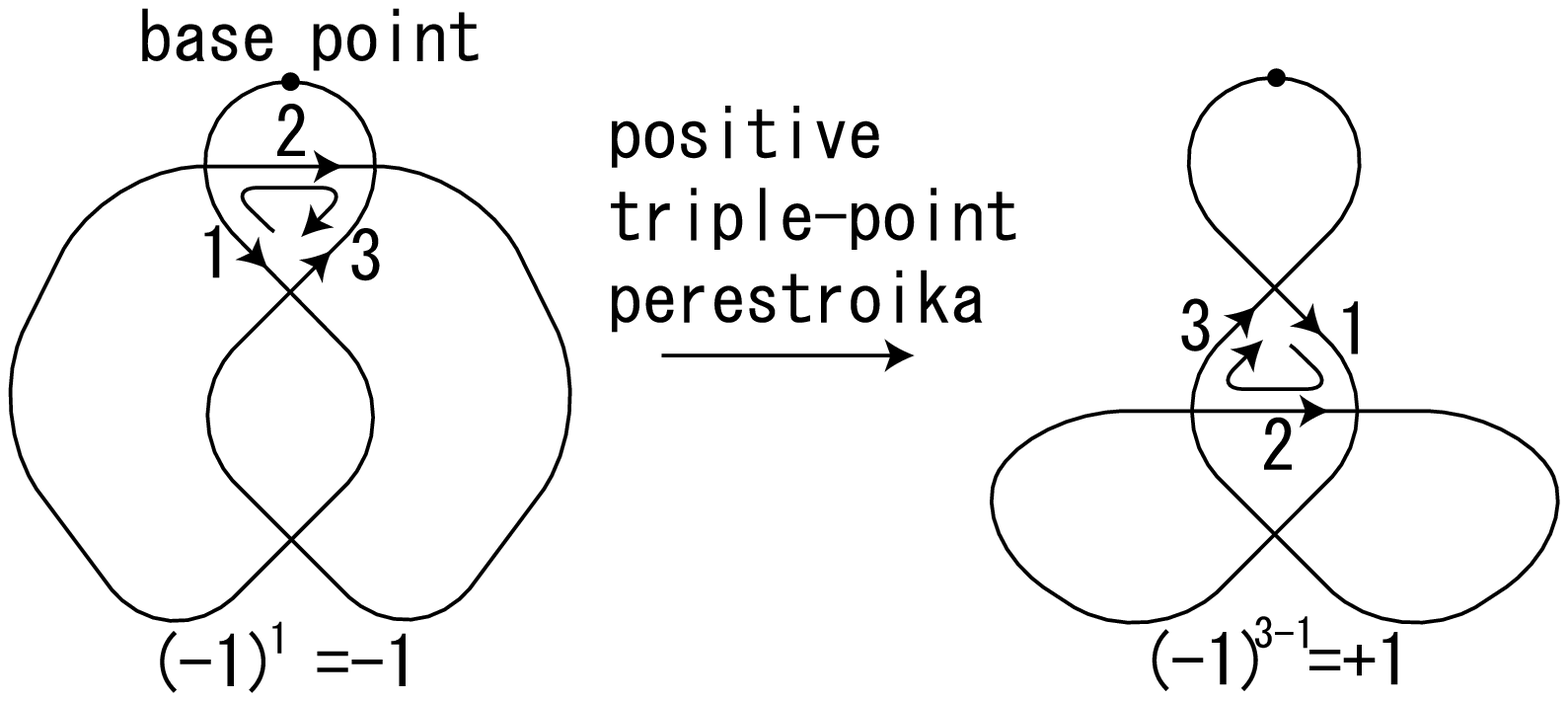}
\end{center}
\caption{}
\label{fig:SignTriple}
\end{figure} 

 When a knot diagram in the plane is deformed by an RII move,
a {\it self-tangency perestroika} occurs
on the underlying plane curve.
 A self-tangency perestroika is called 
{\it positive} (resp. {\it negative}) 
if it creates (resp. {\it deletes}) a bigon face,
and called {\it direct} (resp. {\it inverse})
if the corresponding RII move is matched (resp. unmatched).
 See Figure \ref{fig:perestroika}.
 When two knot diagrams are connected by an RIII move,
then their underlying plane curves are connected
by a {\it triple-point perestroika}.
 The {\it sign} of a triple-point perestroika is determined 
by the sign of the created triangle face
of the plane curve after the perestroika.
 Let $\Delta$ be a triangle face of a plane curve $\Gamma$.
 We take a base point $p$ on $\Gamma$
so that it is disjoint from $\Delta$.
 Let $e_1, e_2, e_3$ be the edges of $\Delta$,
where they are numbered 
so that they appear in this order
when we go around $\Gamma$ once from $p$ to $p$
in the direction of the orientation of $\Gamma$.
 We can orient the boundary circle $\partial \Delta$
so that we meet $e_1$, $e_2$ and $e_3$ in this order
when we go around $\partial \Delta$ in the direction of its orientation.
 Let $q$ be the number of the edges among $e_1$, $e_2$ and $e_3$
on which the orientation induced from $\Gamma$
matches that from $\partial \Delta$.
 Then the {\it sign} of $\Delta$ is defined by $(-1)^q$.
 Note that changing the base point does not affect the sign of $\Delta$.
 It can be easily seen 
that the triangle faces deleted and created by a triple-point perestroika
have opposite signs.
 See Figure \ref{fig:SignTriple},
where an example of a positive triple-point perestroika is described.

\begin{figure}[htbp]
\begin{center}
\includegraphics[width=10cm]{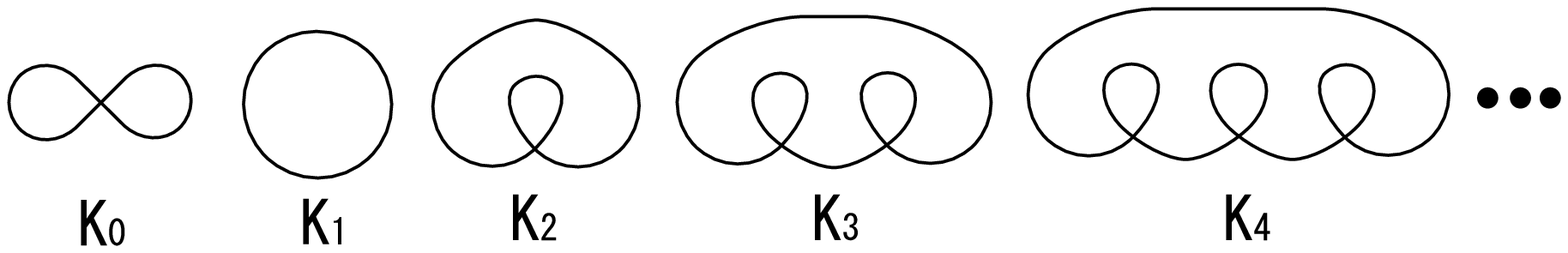}
\end{center}
\caption{}
\label{fig:Ki}
\end{figure}

 Arnold showed in \cite{A}
that there are invariants $J^+, J^-$ and $St$ for plane curves
as below.
 See also \cite{P},
where Polyak gave formulae 
calculating $J^+, J^-$ and $St$ via Gauss diagrams.

\begin{definition}\label{definition:Arnold}
\begin{enumerate}
\item[(1)] 
$J^+$, $J^-$ and $St$ are independent 
of the choice of orientation of a plane curve.
\item[(2)] 
$J^+$ does not change 
under an inverse self-tangency or triple-point perestroika 
but increases by $2$ under a positive direct self-tangency perestroika.
\item[(3)] 
$J^-$ does not change 
under a direct self-tangency or triple-point perestroika 
but decreases by $2$ under a positive inverse self-tangency perestroika.
\item[(4)] 
$St$ does not change under a self-tangency perestroika
but increases by $1$ under a positive triple-point perestroika.
\item[(5)] 
For the plane curves $K_i$, $i \in {\Bbb N} \cup \{ 0 \}$ 
depicted in Figure \ref{fig:Ki},\\
$J^+(K_0)=0, \ J^-(K_0)=-1, \ St(K_0)=0$, \\
$J^+(K_{i+1})=-2i, \ J^-(K_{i+1})=-3i, \ St(K_{i+1})=i$, where $i \ge 0$.
\end{enumerate}
\end{definition}

 Note that $K_i$ has Whitney index (or widing number) $+i$ or $-i$
according to a choice of orientation of $K_i$.
 (The Whitney index of a plane curve $\Gamma$ is calculated as below.
 Smoothing (cutting and pasting) all the double points with respect to the orientation of $\Gamma$,
we obtain disjoint union of oriented circles with no double points.
 Then the index is the number of circles oriented anti-clockwise
minus the number of circles oriented clockwise.)
 Whitney showed in \cite{W} 
that two plane curves are connected by a smooth homotopy
if and only if they are of the same index.
 Two homotopic plane curves are connected 
by a sequence of self-tangency perestroikas and triple-point perestroikas.
 
 Aicardi studied the invariant $J^+/2 + St$ in \cite{Ai}. 
 We regard the $2$-sphere $S^2$ as ${\Bbb R}^2 \cup \{ \infty \}$.
 Then $J^+ /2 + St$ and also $J^- /2 +St$ give invariants for generic spherical curves.
 In fact, they does not depend 
on the choice of the point at infinity $\infty$ in $S^2 -\Gamma$,
where $\Gamma$ is a generic spherical curve.
 This fact is also implied 
by the formulae $J^+ = J^- + n$ 
and $J^+(\Gamma) + 2St(\Gamma) = -2 <B_4, G_{\Gamma}>$ in 4.3 in \cite{P},
where $n$ denotes the number of double points
and the term $<B_4, G_{\Gamma}>$ depends 
only on the Gass diagram $G_{\Gamma}$ of the spherical curve $\Gamma$.
 We obtain $J^- (\Gamma)/2 + St(\Gamma) = -<B_4, G_{\Gamma}> -n/2$ from these formulae.
 (Note that 
$St(\Gamma)$ should be equal to
$\frac{1}{2}<-B_2+B_3+B_4> 
+
\displaystyle\frac{n-1}{4} + \displaystyle\frac{ind(\Gamma)^2}{4}$
in Theorem 1 in \cite{P}.)

\begin{figure}[htbp]
\begin{center}
\includegraphics[width=5cm]{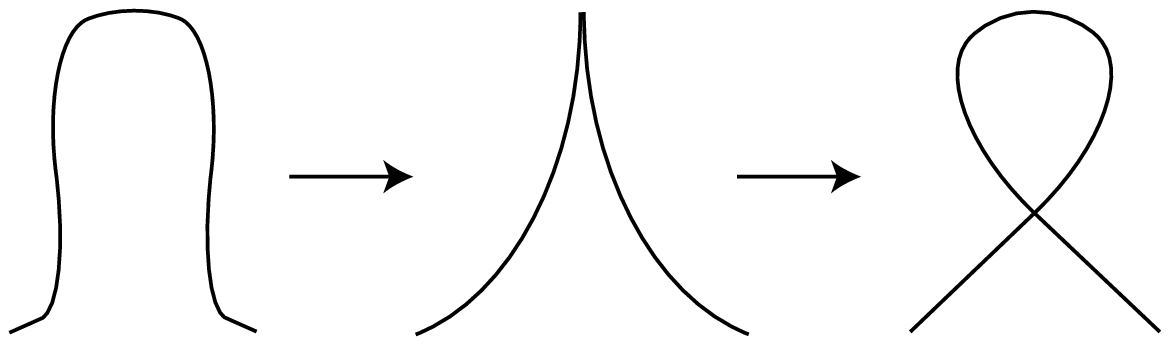}
\end{center}
\caption{}
\label{fig:cusp}
\end{figure}

 We consider changes of $J^+ /2 + St$ and $J^- /2 + St$
under a cusp perestroika as shown in Figure \ref{fig:cusp},
which can be lifted to an RI move on a knot diagram.
 The next proposition is probably well-known.
 It follows easily from Definition \ref{definition:Arnold}
and the above formulae in \cite{P}. 
 See also Proposition 3 in \cite{HN1}.

\begin{proposition}\label{proposition:J/2+St}{\rm (well-known)}
\begin{enumerate}
\item[(1)]
$J^+ /2 + St$ does not change 
under a cusp or inverse self-tangency perestroika,
but increases by $1$ under a positive direct self-tangency perestroika
or positive triple-point perestroika.
\item[(2)] 
$J^- /2 + St$ does not change
under a direct self-tangency perestroika,
but decreases by $1/2$ 
under a cusp perestroika creating a monogon,
by $1$ under a positive inverse self-tangency perestroika
or negative triple-point perestroika.
\end{enumerate}
\end{proposition}

 We can obtain lower bounds 
for the minimal number of Reidemeister moves 
connecting two knot diagrams in $S^2$
representing the same knot
by calculating these invariants 
of underlying spherical curves of the knot diagrams.
 In fact, 
as Hass and Nowik showed in Section 4 in \cite{HN1}, 
the cowrithe of a knot diagram is equal to $-\{ J^+ /2 + St - 4c_2 \}$,
where $c_2$ is the coefficient of $x^2$ of the Conway polynomial of the knot.
 Hence the estimation of the number of Reidemeister moves
by the cowrithe coincides with that by $J^+ /2 + St$.

\begin{figure}[htbp]
\begin{center}
\includegraphics[width=4cm]{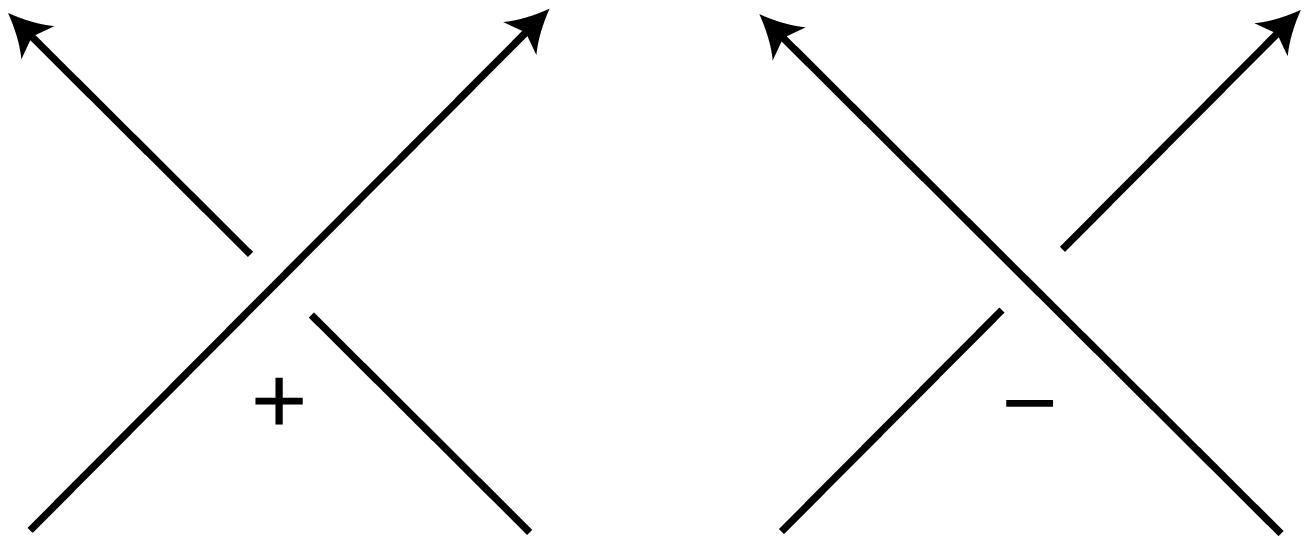}
\end{center}
\caption{}
\label{fig:sign}
\end{figure}

 The invariant $J^- /2 + St$ is sensitive 
to inverse self-tangency perestroikas,
and hence to unmatched RII moves.
 However, it reacts by halves to cusp moves (or RI moves).
 Hence we consider 
$J^-(\bar{D}) /2 + St(\bar{D}) \pm w(D)/2$,
where $D$ is a knot diagram in $S^2$,
$\bar{D}$ the underlying spherical curve,
and $w(D)$ the writhe of $D$.
 The {\it writhe} of a knot diagram $D$
is the sum of signs of all the crossings of $D$,
where the sign of a crossing is defined as shown in Figure \ref{fig:sign}.

 We call an RIII move {\it positive} (resp. {\it negative})
if it causes a positive (resp. negative) triple-point perestroika 
on the underlying spherical curve.

 The next theorem follows easily from Proposition \ref{proposition:J/2+St}
since the writhe does not change under an RII or RIII move
and increases (resp. decreases) by $1$ 
under an RI move creating a positive (resp. negative) crossing.

\begin{theorem}\label{theorem:J^-/2+St+-w/2}
$J^- /2 + St + w/2$ (resp. $J^- /2 + St - w/2$) does not change 
under an RI move creating a positive (resp. negative) crossing
or matched RII move, 
but decreases by $1$ 
under an RI move creating a negative (resp. positive) crossing,
unmatched RII move creating a bigon face
or negative RIII move.
\end{theorem}

 The two formulae $J^+ = J^- + n$ in Section 4.3 in \cite{P},
and $x = 4c_2 - (J^+/2 + St)$ in Section 4 in \cite{HN1}
together imply
$x + n/2 \mp w/2 = 4c_2 - (J^-/2 + St \pm w/2)$,
where $x$ is the cowrithe. 
 Hence we obtain the next corollary.

\begin{corollary}
$x + n/2 - w/2$  (resp. $x + n/2 + w/2$) does not change 
under an RI move creating a positive (resp. negative) crossing
or matched RII move, 
but increases by $1$ 
under an RI move creating a negative (resp. positive) crossing,
unmatched RII move creating a bigon face
or negative RIII move.
\end{corollary}

 Note that $n/2 + w/2$ (resp. $n/2 - w/2$)
does not change 
under an RI move creating a negative (resp. positive) crossing,
but increases by $1$ 
under an RI move creating a positive (resp. negative) crossing.

\section{Minimal sequence of Reidemeister moves}

 We prove Theorem \ref{theorem:Dn} in this section.
 In the course of the proof, 
we obtain the next proposition.
 A knot diagram of the unknot rarely has the cowrithe with positive value.
 In fact, any knot diagram of the unknot with $8$ or less number of crossings has negative cowrithe.
 Note that $c_2 (D_n) = 0$ since $D_n$ represents the unknot.

\begin{proposition}$\ $
\newline
$J^-(\bar{D_n})/2 + St(\bar{D_n}) - w(D_n)/2 
= -({}_n C_1 + {}_n C_2 + {}_n C_3) = -n(n^2+5)/6$
\newline
$J^+(\bar{D_n})/2 + St(\bar{D_n}) = -x(D_n) = -{}_n C_3 = -(n-2)(n-1)n/6$
\end{proposition}

\begin{proof}[Proof of Theorem \ref{theorem:Dn}]
 We first sketch the proof very roughly.
 The trivial knot diagram 
is the unit circle $S^1$ in $S^2 \cong {\Bbb R}^2 \cup \{ \infty \}$,
and it is the union of $n$ arcs $\gamma_1, \gamma_2, \cdots, \gamma_n$,
where $\gamma_i$ is given
by the equation below.
\begin{center}
$r = 1$, $\ (2(i-1)\pi/n \le \theta \le 2i\pi/n)$
\end{center}
 We apply RI moves creating a positive crossing 
${}_n C_1 = n$ times to the trivial knot diagram
so that each subarc $\gamma_i$ is deformed 
into a kink $\lambda_i$ with a positive crossing and a small monogon,
and so that the circle is deformed into a knot diagram
with the curve $K_n$ in Figure \ref{fig:Ki} 
being the underlying spherical curve.
 Let $I$ be a subset of $\{ 1,2, \cdots, n \}$,
and $K(I)$ the knot digram 
obtained from the circle $S^1$
by replacing $\gamma_i$ by $\lambda_i$ for all $i \in I$.
 We perform ${}_n C_2 = (n-1)n/2$ unmatched RII moves creating a bigon
and ${}_n C_3 = (n-2)(n-1)n/6$ negative RIII moves
so that 
the monogons of the kinks are enlarged,
that 
$\lambda_j$ goes over $\lambda_i$ if $i < j$,
that 
$K(\{ i,j \})$ is deformed to a diagram equivalent to $D_2$ 
for every pair of two distinct numbers $i,j$ in $\{ 1,2, \cdots, n \}$,
and that 
$K(\{ i,j,k\})$ is deformed to a diagram equivalent to $D_3$
for every triple of three 
distinct numbers $i,j,k$ 
in $\{ 1,2, \cdots, n \}$.
 Then the resulting knot diagram is $D_n$.
 This deformation and Theorem \ref{theorem:J^-/2+St+-w/2} show
that $J^-(\bar{D_n})/2+St(\bar{D_n})-w(D_n)/2 = -{}_n C_1 -{}_n C_2 -{}_n C_3$,
and the threorem follows.

 Now we describe the proof of the theorem in detail.
 Let $\mu_i$ be the subarc of the diagram $D_n$ given by the formula below.
\begin{center}
$r = 2 + \cos (n \theta /(n+1))$, 
$\ \ (2(i-1)(n+1)\pi/n \le \theta \le 2i(n+1)\pi/n)$
\end{center}
 Each arc $\lambda_i$ is going to be deformed to $\mu_i$.
 For a subset $I$ of $\{ 1,2, \cdots, n \}$,
let $D(I)$ denote the knot digram 
obtained from the circle $S^1$
by replacing $\gamma_i$ by $\mu_i$ for all $i \in I$.

\begin{figure}[htbp]
\begin{center}
\includegraphics[width=11cm]{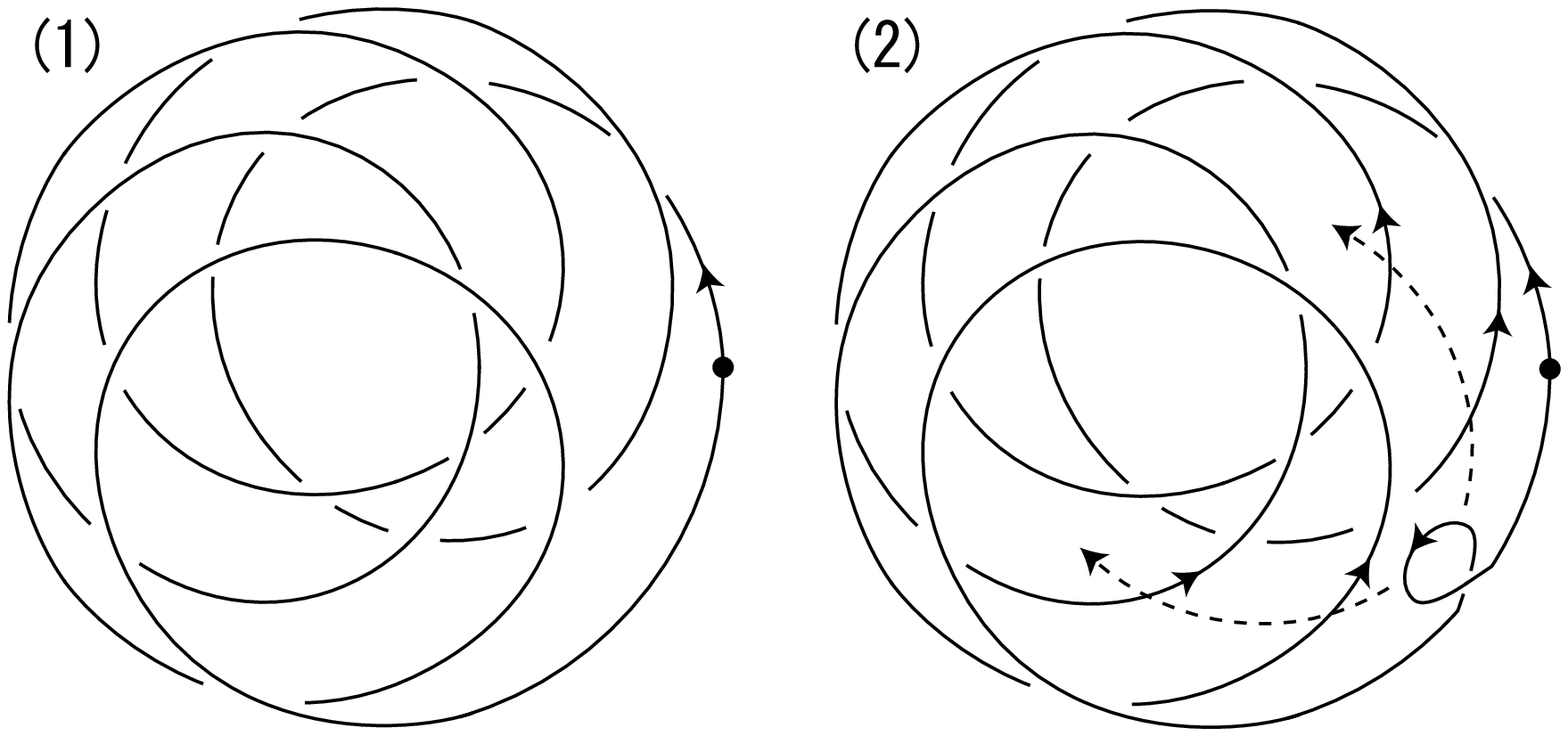}
\end{center}
\caption{}
\label{fig:deform1}
\end{figure} 

\begin{figure}[htbp]
\begin{center}
\includegraphics[width=5cm]{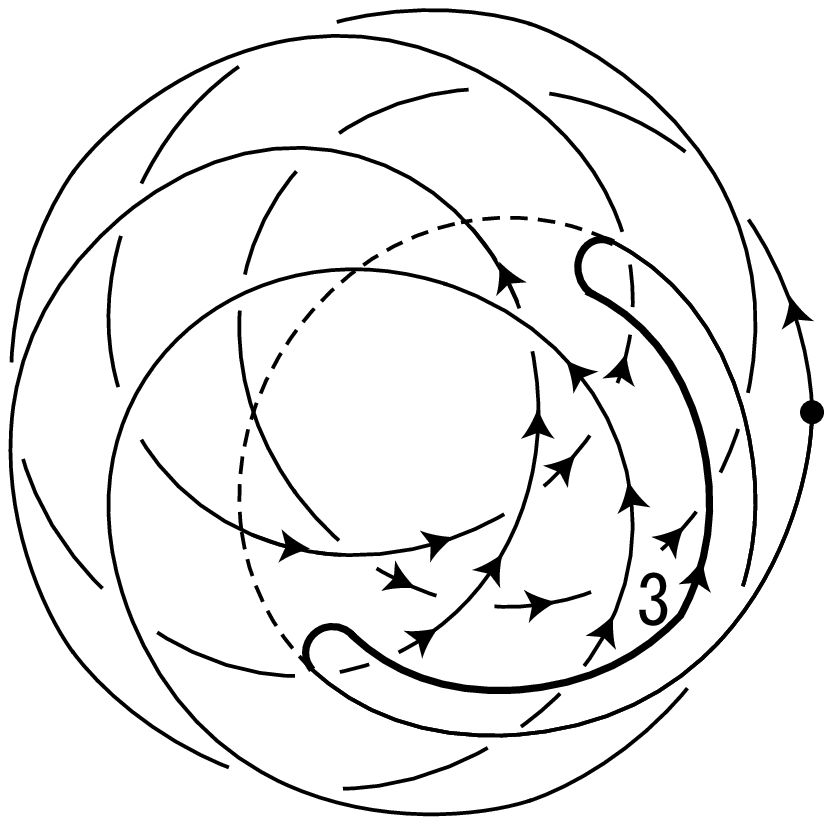}
\end{center}
\caption{}
\label{fig:deform2}
\end{figure} 

\begin{figure}[htbp]
\begin{center}
\includegraphics[width=11cm]{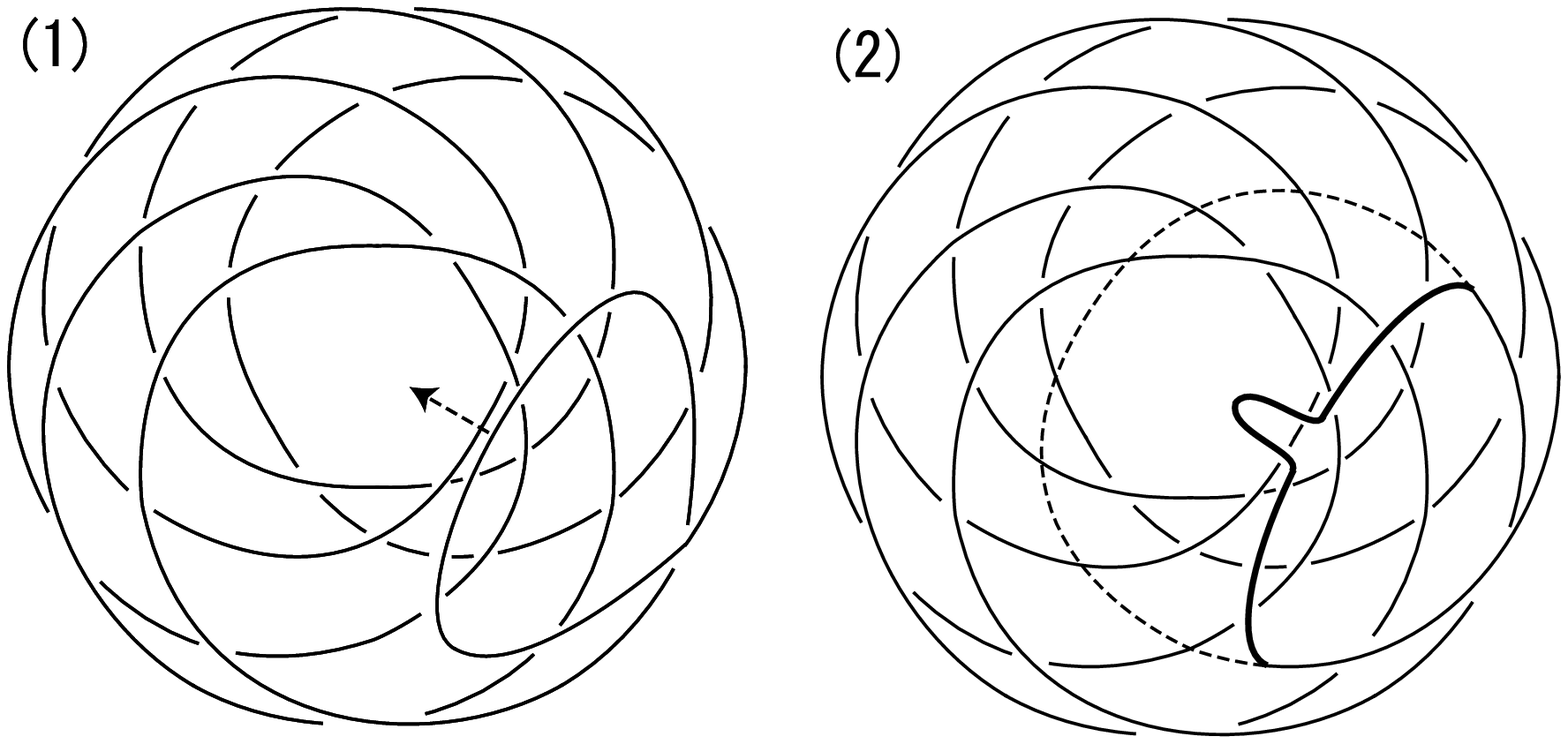}
\end{center}
\caption{}
\label{fig:deform3}
\end{figure} 

 The theorem is proved by an induction on $n$.
 In the case of $D_3$, 
the theorem can be easily confirmed.
 We assume that the theorem holds for $D_{n-1}$
and consider the case of $D_n$.
 Note that the diagram $D(\{ 1,2, \cdots, n-1 \})$ is equivalent to $D_{n-1}$.
 See Figure \ref{fig:deform1}(1).
 We begin with $D(\{ 1,2,\cdots, n-1 \})$,
and deform the subarc $\gamma_n$ to obtain the diagram $D_n$.
 First, we apply an RI move to $\gamma_n$ 
to create the kink $\lambda_n$ with a positive crossing.
 See Figure \ref{fig:deform1}(2).
 We enlarge the monogon bounded by $\lambda_n$.
 Let $\lambda_n$ keep on denoting
the subarc of the knot diagram obtained from $\lambda_n$
by the deformation below.
 We denote by $R(i)$
the RII move between the arc $\mu_i$ and $\lambda_n$,
and by $R(i,j)$
the RIII move on the arcs $\mu_i$, $\mu_j$ and $\lambda_n$.
 The first enlargement of the monogon bounded by $\lambda_n$
is done by the sequence of RII moves
$R(1), R(2), \cdots, R(k)$ and $R(n-1), R(n-2), \cdots, R(\ell)$,
where $k=(n-1)/2$ and $\ell = k+1$ when $n$ is odd, 
and $k=(n-2)/2$ and $\ell = k+2$ when $n$ is even.
 See Figure \ref{fig:deform1}(2) and Figure \ref{fig:deform2}.
 These RII moves are performed 
along the arcs parallel to $\mu_1$ and $\mu_{n-1}$
as shown in Figure \ref{fig:deform1}(2).
 Then, as in Figure \ref{fig:deform2},
we deform the arc drawed in a bold line to that in a broken line.
 Precisely, we first perform RIII moves
$R(1,n-1)$ 
along subarcs of $\mu_1$ and $\mu_{n-1}$, 
$R(1,n-2)$, $R(2,n-1)$, $R(2,n-2)$ 
along subarcs of $\mu_2$ and $\mu_{n-2}$, 
$R(1,n-3)$, $R(3,n-1)$, $R(2,n-3)$, $R(3,n-2)$, $R(3,n-3)$ 
along subarcs of $\mu_3$ and $\mu_{n-3}$, 
$R(1,n-4)$, $R(4,n-1)$, $R(2,n-4)$, $R(4,n-2)$, 
$R(3,n-4)$, $R(4,n-3)$, $R(4,n-4)$ 
along subarcs of $\mu_4$ and $\mu_{n-4}$, 
$\cdots$, 
$R(1,n-k)$, $R(k,n-1)$, $R(2,n-k)$, $R(k,n-2)$, $\cdots$, 
$R(k-1,n-k)$, $R(k,n-(k-1))$, $R(k,n-k)$
along subarcs of $\mu_k$ and $\mu_{n-k}$. 
 See Figure \ref{fig:deform3}(1).

 When $n$ is odd,
we further perform RIII moves
$R(k-1,k)$, $R(k-2,k)$, $\cdots$, $R(1,k)$
along a subarc of $\mu_k$,
$R(k+1,k+2)$, $R(k+1,k+3)$, $\cdots$, $R(k+1,n-1)$
along a subarc of $\mu_{k+1}$,
$R(k-2,k-1)$, $R(k-3,k-1)$, $\cdots$, $R(1,k-1)$
along a subarc of $\mu_{k-1}$,
$R(k+2,k+3)$, $R(k+2,k+4)$, $\cdots$, $R(k+2,n-1)$
along a subarc of $\mu_{k+2}$,
$\cdots$,
$R(1,2)$
along a subarc of $\mu_2$,
$R(n-2,n-1)$
along a subarc of $\mu_{n-2}$.
 Thus $D_n$ is obtained.

 When $n$ is even, we do the RII move $R(k+1)$.
 See Figure \ref{fig:deform3}.
 Then, we apply RIII moves
$R(k,k+1)$, $R(k-1,k+1)$, $\cdots$, $R(1,k+1)$
along a subarc of $\mu_{k+1}$,
$R(k+1,k+2)$, $R(k+1, k+3)$, $\cdots$, $R(k+1, n-1)$
along a subarc of $\mu_{k+1}$,
$R(k-1,k)$, $R(k-2,k)$, $\cdots$, $R(1,k)$
along a subarc of $\mu_k$,
$R(k+2,k+3)$, $R(k+2,k+4)$, $\cdots$, $R(k+2,n-1)$
along a subarc of $\mu_{k+2}$,
$R(k-2,k-1)$, $R(k-3,k-1)$, $\cdots$, $R(1,k-1)$
along a subarc of $\mu_{k-1}$,
$R(k+3,k+4)$, $R(k+3,k+5)$, $\cdots$, $R(k+3,n-1)$
along a subarc of $\mu_{k+3}$,
$\cdots$,
$R(1,2)$
along a subarc of $\mu_2$,
$R(n-2,n-1)$
along a subarc of $\mu_{n-2}$.
 Thus we obtain $D_n$.

 In both cases,
we have performed a single RI move,
${}_{n-1} C_1 = n-1$ RII moves 
and ${}_{n-1} C_2 = (n-2)(n-1)/2$ RIII moves
to deform $\gamma_n$ to $\mu_n$.
 (In fact, the RII move $R(i)$ has been performed
for every integer $i$ with $1 \le i \le n-1$,
and the RIII move $R(i,j)$ has been performed
for every pair of integer $i,j$ with $1 \le i < j \le n-1$.)
 We do ${}_{n-1} C_m$ Reidemeister moves of the $m$-th type to obtain $D_{n-1}$.
 Hence the formula
${}_{n-1} C_m + {}_{n-1} C_{m-1} = {}_n C _m$
implies the theorem. 
\end{proof}



\bibliographystyle{amsplain}

\medskip

\noindent
All authors: 
Department of Mathematical and Physical Sciences,
Faculty of Science, Japan Women's University,
2-8-1 Mejirodai, Bunkyo-ku, Tokyo, 112-8681, Japan.

\vspace{2mm}

\noindent
Chuichiro Hayashi:
hayashic@fc.jwu.ac.jp

\noindent
Miwa Hayashi:
miwakura@fc.jwu.ac.jp

\noindent
Minori Sawada:
m0716060sm@ug.jwu.ac.jp

\noindent
Sayaka Yamada:
m0716131ys@ug.jwu.ac.jp

\end{document}